\documentclass[sort&compress]{myarticle}
\usepackage{graphicx}        
\usepackage{multicol}        
\usepackage{amssymb,amsbsy}
\usepackage{amsthm,float}

\newcommand{\cA}{{\cal A}}
\newcommand{\cB}{{\cal B}}

\newcommand{\cE}{{\cal E}}

\newcommand{\ZZ}{\mathbb{Z}}
\newcommand{\RR}{\mathbb{R}}
\newcommand{\NN}{\mathbb{N}}
\newcommand{\CC}{\mathbb{C}}

\newcommand{\Ab}{{\boldsymbol{A}}}
\newcommand{\Bb}{{\boldsymbol{B}}}
\newcommand{\Cb}{{\boldsymbol{C}}}
\newcommand{\Db}{{\boldsymbol{D}}}
\newcommand{\Eb}{{\boldsymbol{E}}}
\newcommand{\Fb}{{\boldsymbol{F}}}
\newcommand{\Gb}{{\boldsymbol{G}}}

\newcommand{\Ib}{{\boldsymbol{I}}}
\newcommand{\Jb}{{\boldsymbol{J}}}
\newcommand{\Kb}{{\boldsymbol{K}}}
\newcommand{\Lb}{{\boldsymbol{L}}}

\newcommand{\Pb}{{\boldsymbol{P}}}

\newcommand{\Rb}{{\boldsymbol{R}}}

\newcommand{\Ub}{{\boldsymbol{U}}}
\newcommand{\Vb}{{\boldsymbol{V}}}

\newcommand{\ab}{{\boldsymbol{a}}}

\newcommand{\cb}{{\boldsymbol{c}}}
\newcommand{\db}{{\boldsymbol{d}}}
\newcommand{\eb}{{\boldsymbol{e}}}
\newcommand{\fb}{{\boldsymbol{f}}}

\newcommand{\hb}{{\boldsymbol{h}}}

\newcommand{\yb}{{\boldsymbol{y}}}

\newcommand{\BZero}{{\boldsymbol{0}}}

\newtheorem{proposition}{Proposition}
\newtheorem{theorem}{Theorem}
\newtheorem{lemma}{Lemma}

\begin{document}

\begin{frontmatter}

\title{From full rank subdivision schemes to multichannel wavelets: A constructive approach\let\thefootnote\relax\footnotetext{\emph{''NOTICE: this is the authors' version of a work that was accepted for publication 
			in \emph{Wavelets and Multiscale Analysis: Theory and Applications} edited by J. Cohen And A. Zayed, Birkh\"auser, Boston (USA) (ISBN: 978-0-8176-8094-7). 
			Changes resulting from the publishing process, such as peer review, editing, corrections,
			structural formatting, and other quality control mechanisms may not be reflected in this document. Changes may have been made to this work since it was submitted for publication''}}}

\author{Mariantonia Cotronei
}
\address{DIMET, 
Universit\`a Mediterranea di Reggio Calabria,\\Via Graziella loc. Feo di Vito, I-89122 Reggio Calabria, Italy\\
mariantonia.cotronei@unirc.it}

\author{Costanza Conti}

\address{Dipartimento di Energetica  ''Sergio Stecco'', Universit\`a di Firenze, Via  Lombroso 6/17, I-50134 Firenze, Italy,\\
costanza.conti@unifi.it}

\begin{abstract}
In this paper, we  describe some recent results obtained in the context of vector subdivision schemes which possess the so-called full rank property. Such kind of schemes, in particular those which have an interpolatory nature,  are connected to matrix refinable functions generating orthogonal multiresolution analyses for the space of vector-valued signals. Corresponding multichannel (matrix) wavelets can be defined and their construction in terms of a very efficient scheme is given.
Some examples illustrate the nature of these matrix scaling functions/wavelets.
\end{abstract}


\end{frontmatter}

\section{Introduction}
This paper deals with the construction of
the proper ''wavelet'' tools for
the analysis of functions which are \emph{vector-valued}.  Such type of
functions arises naturally in many applications  where the data to be processed are samples of vector-valued functions, even in
several variables (e.g. electroencephalographic measurements, seismic waves, color images).
\par
Up to now, most signal processing methods
handle vector valued data component by component and thus ignore
the possible relationship between several of the components. In
reality, on the other hand, such measurements are usually related
and thus should be processed as ``complete'' vectors.
While it is obvious that for vector valued
signals a standard scalar wavelet analysis does not take into
account the correlations among components, even  an approach based
on the well known \emph{multiwavelets}, introduced in  \cite{GHM}, \cite{Strela} and widely studied (see \cite{Keinert} and the exhaustive references therein),  is only apparently
justified by the ''vectorization'' step required in order to apply  them. Indeed, multiwavelets still generate
multiresolution analyses for the space of scalar-valued
functions.
\par
In \cite{Xia}, \cite{XiaSuter} \emph{ matrix wavelets} have been proposed for the analysis of matrix valued signals while in \cite{BacchelliCotroneiSauer02a} the concept of \emph{multichannel wavelet} (MCW) and \emph{multichannel multiresolution analysis} (MCMRA) have been introduced
for processing vector-valued data. Actually, the two concepts look very similar: the crucial point of both is the existence of a  matrix refinable function, which satisfies the so-called \emph{full rank condition}. However,  on one hand, the matrix analysis scheme proposed in \cite{Xia}, \cite{XiaSuter}  coincides with the applications of a fixed number of MCW schemes, on the other hand, only a very special characterization of filters is proposed, which gives rise  to interpolatory orthonormal matrix refinable functions. Thanks to this special structure, the corresponding wavelet filters are easily found as in the scalar situation. Nevertheless, in \cite{Xia}, \cite{XiaSuter} the authors do not provide either any constructive strategy nor any example. Some examples of $2\times 2$ matrix wavelets are derived in \cite{Walden}, but they strongly rely on the sufficient condition for convergence described in \cite{Xia}, and belong again to a very special class. In passing, we also mention reference \cite{Fowler} since, in spite of the title, it actually contains examples of wavelet filters specifically designed for vector-fields.
\par
In some recent papers \cite{ContiCotroneiSauer08},
\cite{ContiCotroneiSauer07S}, \cite{ContiCotroneiSauer07}  we investigated  \emph{full rank interpolatory schemes} and showed their connection to matrix refinable functions and multichannel wavelets.
In particular we proved  that there is  a connection between \emph{cardinal} and \emph{orthogonal} matrix refinable functions which is based on the \emph{spectral factorization} of the (positive definite) symbol related to the cardinal function, thus, as a by-product, giving a concrete way to obtain orthogonal matrix scaling functions. The corresponding  multichannel wavelets, whose existence is proved in  \cite{BacchelliCotroneiSauer02a}, can be found in terms of a symbol which satisfies  \emph{matrix quadrature mirror filter} equations.
\par
In this paper, we show how  the solution of matrix quadrature mirror filter equations  can be found. In particular we give an efficient and constructive scheme which makes use again of spectral factorization techniques and of a matrix completion algorithm based on the resolution of generalized Bezout identities.
\par

\smallskip The paper is organized as follows. In Section~2 all the needed notation and definitions are set.
 Some introductory results on full rank vector subdivision schemes are also presented  and their connection to
orthogonal full rank refinable functions is shown. In Section~3, the concepts of
MCMRA and MCW are revised while their orthogonal counterparts are considered next in Section~4. 
Section~5 deals with the construction of
 multichannel wavelets and an explicit algorithm for such construction is presented.
 In Section~6, starting from some interpolatory full
rank positive definite symbols, the associated multichannel  wavelets
with two and three channels are constructed and shown. 
Some conclusions are derived in Section~7.

\section{Notation and basic facts on vector subdivision schemes}
For $r, s \in \NN$ we write a
 matrix $\Ab \in \RR^{r
\times s}$ as $\Ab = ( a_{jk}:$ \ $j = 1,\dots,r,\ k =1,\dots,s)$ and denote by $\ell^{r \times s} \left( \ZZ
\right)$ the space of all $r \times s$--matrix valued bi--infinite
sequences, $\cA = \left( \Ab_j\in \RR^{r \times s} \;:\; j
  \in \ZZ \right)$.
For notational simplicity we write $\ell^r (\ZZ)$ for $\ell^{r
\times 1} \left( \ZZ \right)$ and $\ell (\ZZ)$ for $\ell^{1}
\left( \ZZ \right)$ and denote vector sequences by lowercase
letters.
\par
By  $\ell_2^{r \times s} \left( \ZZ
\right)$ we denote the space of all sequences which have finite $2$ norm defined as
\begin{equation}
  \label{eq:ANormDef}
 \left\| \cA \right\|_2:=
     \left(\sum_{j\in \ZZ}\left| \Ab_j \right|^2_2\right)^{1/2},
    \end{equation} where
    $\left| \cdot \right|_2$ denotes the
$2$ operator norm for $r\times s$ matrices.
\par
Moreover, $L^{r\times s}_2(\RR)$
 will denote the
Banach space of all $r \times s$--matrix valued functions on $\RR$
with components in $L_2(\RR)$ and norm
\begin{equation}\label{eq:FNormDef}
\|\Fb\|_2=\left(\sum_{k=1}^{s}\sum_{j=1}^{r}\int_{\RR}|F_{jk}(x)|^2\, dx
\right)^{1/2}.
\end{equation}
\par
For a matrix function and a matrix sequence of suitable sizes we introduce the
\emph{convolution} ``$*$'' defined as, 
\begin{eqnarray*}
\left( \Fb * \cB \right):= \sum_{k \in \ZZ} \Fb(\cdot -k)
\, \Bb_k.
\end{eqnarray*}
For two matrix functions $\Fb,\ \Gb \in L^{r\times s}_2(\RR)$, their inner product is
defined as
\begin{eqnarray*}
\langle \Gb, \Fb\rangle=\int_{\RR} \Gb^H(x)\Fb(x)\, dx.
\end{eqnarray*}
 and we have
$$\|\Fb\|_2= \left(\mbox{trace}\, \langle \Fb, \Fb\rangle \right)^\frac 12.$$
\par
Let now $\cA$ be  a \emph{matrix
sequence} $ \cA = \left( \Ab_j \in \RR^{r \times r} \;:\; j \in
\ZZ \right)\in \ell^{r\times r}(\ZZ)$. Let $c$ be any
bi--infinite \emph{vector sequence} $ c = \left( \cb_j \in \RR^r
\;:\; j \in \ZZ \right) $. The \emph{vector subdivision operator} $S_{\cA}$,
based on the matrix mask $\cA$,   acts on $c$ by means
of
\begin{eqnarray*}
S_\cA c = \left( \sum_{k \in \ZZ} \Ab_{j-2k} \, \cb_k \;:\; j \in
\ZZ \right).
\end{eqnarray*}
\par
 A \emph{vector subdivision scheme } consists of
iterative applications of a vector subdivision operator starting
from an initial vector sequence $c\in \ell^r(\ZZ)$, namely:
$$\begin{array}{ll}
c^0&:=c \\
c^{n}&:=S_\cA c^{n-1}=S^n_\cA c^{0}\,\quad n\ge 1.
\end{array}$$
It is {\em $L_2$-convergent} if
 there exists a
 \emph{matrix-valued function} $\Fb\in L^{r\times r}_2(\RR)$
such that
\begin{equation}\label{def:conv_p} \lim_{n \to
\infty} 2^{-n/2}\| \mu^n(\Fb)-S_\cA^n \delta \Ib
  \|_2 = 0,
\end{equation}
where the mean value operator at level $n\in \NN$ is defined
 as
$$
 (\mu^n(\Fb))_j:=2^n\int_{2^{-n}(j+[0,1])}\Fb(t)dt,\quad j\in \ZZ.
$$
The matrix-valued function  $\Fb$ associated in such a way with a
convergent subdivision scheme is called the \emph{basic limit
function} and it is {\em refinable} with respect to $\cA$, that is
$$  \Fb = \Fb * \Ab \left( 2 \cdot \right) = \sum_{j \in \ZZ} \Fb \left( 2 \cdot
  - j \right) \, \Ab_j.
$$
\par
The {\em symbol} of a subdivision scheme $S_{\cA}$ or of the  mask
$\cA$ is defined as:
$$\Ab (z) = \sum_{j \in \ZZ} \Ab_j \, z^j, \qquad z \in \CC \setminus
\{0\},
$$
and the {\em subsymbols}, $\Ab_\varepsilon(z) = \sum_{j \in \ZZ}
\Ab_{\varepsilon+2j }\, z^j,\ z\in \CC \setminus \{ 0 \}, \
\varepsilon
  \in \{0,1\},$ are related to the symbol by
\begin{equation}\label{subsymbols}\Ab(z)=\Ab_0(z^2)+z\, \Ab_1(z^2), \quad z \in \CC \setminus
\{0\}.\end{equation}
Next, we define  the two $r\times r$ matrices
$$
\Ab_\varepsilon :=  \Ab_\varepsilon(1)=\sum_{j \in \ZZ}
\Ab_{\varepsilon+2j},\quad \varepsilon\in \{0,1\}
$$
and their joint $1$--eigenspace
$$
\cE_\Ab :=\left\{ \yb \in \RR^r \;:\; \Ab_0\, \yb = \Ab_1 \, \yb =
\yb
  \right\}.
$$
The \emph{rank} of the mask $\cA$ or of the  subdivision scheme
$S_\cA$ is the number
$$  R(\cA) := \dim \cE_\Ab
$$
satisfying $1 \le R(\cA) \le r$ for convergent schemes. In
particular, $S_\cA$ is said to be \emph{of full rank} if $R(\cA) =
r$.
\par
The following proposition gives equivalent characterizations of full rank schemes
\cite{ContiCotroneiSauer08}
\begin{proposition}
Let $ S_\cA$ be a vector subdivision scheme with symbol $\Ab(z)$.
Then the following statements are equivalent:
\begin{itemize}
  \item[\emph{i) } ]  $S_\cA$ is of full rank;
   \item[\emph{ii) } ]
  the symbol satisfies $\Ab(1) = 2\Ib$ and $\Ab(-1) = \BZero,$
    or, equivalently, there exists a matrix mask $\cB\in \ell^{r\times r}(\ZZ)$
    such that  $\Ab(z)=(z+1)\Bb(z)$, with $\Bb(1)=\Ib$ (scalar-like factorization);
    \item[\emph{iii) } ] the scheme preserves all vector constant data, i.e
    whenever $\cb_j = \db$, $j \in \ZZ$, for
some $\db \in \RR^r$, then also $\left( S_\cA c \right)_j = \db,
\quad j \in \ZZ$;
 \item[\emph{iv) } ] the basic limit function $\Fb$ is a
partition of the identity, i.e. $\displaystyle{\sum_{j\in \ZZ} \Fb(\cdot -j )=\Ib}$.
\end{itemize}
\end{proposition}

\par
A full rank vector subdivision scheme is \emph{interpolatory} if
it is characterized by the property $$ \left( S_\cA c \right)_{2j}
= \cb_j,\ j \in \ZZ, $$ or, equivalently,
$$\Ab_{2j} = \delta_{j0} \, \Ib, \qquad  j \in \ZZ.$$
Its associated basic limit function, whenever it exists, turns out to be {\em cardinal},
i.e.,$$\Fb (j) = \delta_{j0} \, \Ib, \qquad  j \in \ZZ.$$
The
connection between interpolatory and full rank subdivision schemes
is expressed by the following result \cite{ContiCotroneiSauer08}:

\begin{proposition}
An interpolatory subdivision scheme with symbol $\Ab(z)$ is of
full rank if and only if  $\ \Ab (1) = 2\Ib$.
\end{proposition}

\section{Multichannel multiresolution analysis and multichannel wavelets }
As in the scalar situation, a \emph{multichannel multiresolution analysis} (MCMRA) in the space $L_2^r\left( \RR \right)$ of square integrable vector valued functions can be defined b a nested sequence $\cdots\subset V_{-1}\subset V_0 \subset V_1 \subset \cdots $ of closed
  subspaces of $L_2^r\left( \RR \right)$
  with the properties that
  \begin{enumerate}
  \item they are {\em shift invariant}
    $$
    \hb \in V_j \qquad \Longleftrightarrow \qquad \hb \left( \cdot + k
    \right) \in V_j, \quad k \in \ZZ;
    $$
  \item they are {\em  scaled} versions of each other
    $$
    \hb \in V_0 \qquad \Longleftrightarrow \qquad \hb \left( 2^j \cdot
    \right) \in V_j;
    $$
      \item $\displaystyle{\bigcap_{j\in\ZZ}} V_j=\{[0,\,0\,\ldots\,0]^T\},\quad \displaystyle{\overline{\bigcup_{j\in\ZZ} V_j}=L_2^r\left( \RR \right)};$
  \item the space $V_0$ is generated by the integer translates of $r$
     {\em  function vectors}, that is there exist $\fb^i=\left(f^i_1,\ldots,f^i_r\right)^T\in L^r_2(\RR)$, $i=1,\ldots,r$, such that

\begin{enumerate}
    \item[\emph{i)}]   $\displaystyle{
    V_0 = \mbox{\rm span} \, \left\{ \fb^i(\cdot-k) \,:\; k\in \ZZ,\,i=1,\ldots,r \right\}},
    $
    \item[\emph{ii)}] there exist two constants $K_1\le K_2$ such that
\begin{equation}\label{def:stab}     K_1\left(\sum_{i=1}^r \sum_{k \in
          \ZZ} \left| c^i_{k} \right|^2 \right)^{1/2} \le \left\| \sum_{i=1}^r \sum_{k \in \ZZ} c^i_{k} \fb^j \left(
\cdot
          - k \right) \right\|_2 \le K_2\left(\sum_{j=1}^r \sum_{k \in
          \ZZ} \left| c^i_{k} \right|^2 \right)^{1/2}\end{equation}
          for any $r$ scalar sequences $c^i=\{c^i_k\}\in \ell_2(\ZZ)$, $i=1,\ldots,r$.
\end{enumerate}
  \end{enumerate}
\par
Now, let  $\hb=(h_1,\ldots,h_r)^T\in V_0\subset L^r_2(\RR)$. Then there exist $r$ scalar
sequences
$c^i=\left(c^i_k\in \RR,\ :\ k\in \ZZ\right)\in  \ell_2(\ZZ),\ i=1,\ldots,r$, such that
$$\hb=\sum_{k\in\ZZ}\left( c^1_k \fb^1(\cdot -k )+ \cdots + c^r_k \fb^r(\cdot -k )\right),$$
that is
\begin{eqnarray*}
\left(\begin{array}{c} h_1\\ \vdots \\ h_r\end{array}\right)&=&
\sum_{k\in\ZZ}\left( c^1_k \left(\begin{array}{c} f^1_{1}(\cdot-k)\\
\vdots
\\ f^1_{r}(\cdot-k)\end{array}\right) + \cdots + c^r_k
\left(\begin{array}{c} f^r_{1}(\cdot-k)\\ \vdots \\
f^r_{r}(\cdot-k)\end{array}\right)
\right)\\
&=& \sum_{k\in\ZZ} \underbrace{\left(\begin{array}{ccc} f^{1}_1(\cdot-k) & \cdots & f^r_{1}(\cdot-k) \\
\vdots & & \\
f^1_{r}(\cdot-k) & \cdots &
f^r_{r}(\cdot-k)\end{array}\right)}_{\Fb(\cdot-k)}
\underbrace{\left(\begin{array}{c} c^1_k\\ \vdots \\
c_k^r\end{array}\right)}_{\cb_k}.
\end{eqnarray*}
Thus, any $\hb \in V_0$ can be written as
$\displaystyle{\hb=\Fb*c=\sum_{k\in\ZZ} \Fb(\cdot -k ) \cb_k}$,
where $\Fb \in L_2^{r\times r}(\RR)$ and $c=\left( \cb_k\in \RR^r\
:\ k\in \ZZ\right)\in \ell_2^r(\ZZ)$. Since $\fb^1,\ldots, \fb^r
\in V_0\subset V_1$, we have
$$\fb^i=\sum_{k\in\ZZ} \Fb(2\cdot -k )\ab_k^{i}, \quad i=1,\ldots,r,$$
for some vector sequences
$a^i=\left(
\ab^i_k\in \RR^r,\ :\ k\in \ZZ\right)\in
\ell_2^r(\ZZ)$, $i=1,\ldots,r$.
Thus, the matrix-valued function
$$\Fb=\left(\fb^1|\ldots|\fb^r\right)\in L_2^{r\times r}(\RR)$$
satisfies the \emph{matrix refinement equation}
\begin{equation}\label{eq:RE}
\Fb=\sum_{k\in\ZZ}\Fb (2\cdot-k)\Ab_k
\end{equation}
where $\cA=\left(\Ab_k=\left(\ab^1_k|\ldots|\ab^r_k\right)\in
\RR^{r\times r}\,:\, k\in \ZZ\right)$ is the \emph{refinement mask}.
\par
 Now, since the subdivision scheme $S_{\cal A}$ applied to the initial sequence $\delta\Ib$
converges to a matrix function with $r$
stable columns it follows that $\cE_{\cA}=\RR^r$ so that
$$\Ab_{\varepsilon}=\sum_{j\in \ZZ}\Ab_{\varepsilon-2j}=\Ib,\ \varepsilon\in
\{0,1\},
$$ i.e. $\Ab(1)=2\Ib,\ \Ab(-1)=\BZero$, which is the full rank property of the mask  $\cA$.
\par
 The "joint stability" condition (\ref{def:stab}) on the function vectors $\fb^j,\;
j=1,\ldots,r,$  implies that the function $\Fb$ is stable in the
sense that there exist two constants $K_1\le K_2$ such that
$$   K_1\, \left\|
c\right\|_2\le \left\| \Fb * c \right\|_2 \le K_2\\ \left\|
c\right\|_2, \qquad c
    \in \ell^r_2 (\ZZ).
$$

\par
 Summarizing the described properties, the following important result holds true:
\begin{theorem}
  The subspaces $V_j \subset L_2^{r}(\RR)$, $j \in \ZZ$,
  form an MCMRA if there exists a full rank stable matrix refinable function
  $\Fb \in L_2^{r \times r}(\RR)$ such that
$$    V_j = \left\{ \Fb * c (2^{j}
\cdot )\;:\; c \in \ell^r_2 \left( \ZZ \right)
    \right\}, \qquad j \in \ZZ.
 $$
\end{theorem}
\par
In analogy to the scalar case, we call $\Fb$ \emph{matrix scaling
function}.

\section{Orthogonal multichannel multiresolution analysis}
For many application purposes, \emph{orthogonal MCMRAs} are the most interesting. They are are
equivalently characterized by the following properties:
\begin{enumerate}
    \item the space $V_0$ is generated by  orthonormal  integer translates of the
    function vectors $\fb^j$, $j=1,\ldots,r$;
   \item  the matrix scaling function $\Fb$ is orthonormal  that is
 $$
\langle\Fb,\Fb(\cdot +k)\rangle:=\int_{\RR} \Fb^H (x) \, \Fb (x+k)
\, dx = \delta_{k0}\,\Ib, \qquad k \in \ZZ.
$$
\end{enumerate}

\par
The key for the construction of orthonormal matrix scaling
functions (and the associated multichannel wavelets) is their
connection with interpolatory vector subdivision schemes, as
established in the following Theorem  where the \emph{canonical
spectral factor} $\Ab(z)$ of a positive definite interpolatory
full rank symbol  $\Cb(z)$ is the symbol satisfying  $\Cb (z)
=\frac 12 \Ab^H (z) \Ab (z)$ and $\Ab (1) = 2 \, \Ib,
    \ \Ab (-1) = \BZero$.

\begin{theorem}\cite{ContiCotroneiSauer08}
Let $\Cb(z)$ be the symbol of an interpolatory full rank vector subdivision scheme. Let $\Cb(z)$
be symmetric strictly positive definite for all $z\ne -1$ such
that $|z|=1$. Then the canonical spectral factor $\Ab(z)$ of
$\Cb(z)$ defines a subdivision scheme $S_{\cal A}$ which converges
in $L_2^{r\times r}(\RR)$ to an orthonormal matrix scaling function
$\Fb$.
\end{theorem}

As to the construction of such function, in \cite{ContiCotroneiSauer08}  a simple procedure is given
based on a modified \emph{spectral factorization} of
the \emph{parahermitian matrix} $2\Cb(z)$, which takes into account the presence of some of its zeros on the unitary circle.
\bigskip\par

We can now associate a matrix wavelet to any orthogonal MCMRA in the standard way.
 A matrix function $\Gb \in V_1$ is called a \emph{multichannel wavelet} for the orthogonal MCMRA
  if:
\begin{enumerate}
    \item $
    W_j := V_{j+1} \ominus V_j = \left\{ \Gb * c \, ( 2^{j} \cdot ) \;:\; c \in \ell^r_2 (\ZZ)
    \right\}, \qquad j \in \ZZ;
$
\item $\Gb $ is orthonormal.
\end{enumerate}
\par
As to the existence of multichannel wavelets,
we recall the following result:

\begin{theorem} \cite{BacchelliCotroneiSauer02a}
Suppose that the matrix scaling  function $\Fb \in L_2^{r\times
r}(\RR)$
  has orthonormal integer translates. Then there exists an
  orthonormal wavelet $\Gb \in L_2^{r\times r}(\RR)$ satisfying the two-scale relation
  \begin{equation}\label{eq:twoscalewavelet} \Gb = \sum_{j \in \ZZ} \Fb \left( 2 \cdot
  - j \right) \, \Bb_j\end{equation}
  for a suitable mask $\cB=(\Bb_j:\, j\in \ZZ)$.
 \end{theorem}
\par
Observe that, as in the scalar case, due to the full
rank properties of $\Fb$, the symbol  of the multichannel wavelet
$\Gb$ must possess at least one factor $(z-1)$. This is equivalent
to say that the multichannel wavelet has at least one vanishing
moment.
\par
  It is easy to show that the symbols of an  orthonormal matrix function  $\Fb $ and of the corresponding multichannel
  wavelet $\Gb$ satisfy the
  \emph{orthogonality} (\emph{quadrature mirror filter (QMF)}) \emph{conditions}
 \begin{eqnarray*}
&&\Ab^\sharp (z)\Ab(z)+\Ab^\sharp (-z)\Ab(-z)=4\Ib, \quad |z|=1,\\
&&\Ab^\sharp (z)\Bb(z)+\Ab^\sharp (-z)\Bb(-z)=\BZero, \quad \ |z|=1,\\
&&\Bb^\sharp (z)\Bb(z)+\Bb^\sharp (-z)\Bb(-z)=4\Ib, \quad |z|=1,
 \end{eqnarray*}
where $\Ab^\sharp(z):= \Ab^T(z^{-1})$ which means $\Ab^\sharp(z):= \Ab^H(z)$ whenever
 $|z|=1$.
 \par
The QMF can be written in a concise form by the condition
$\Ub^\sharp(z)\Ub(z)=\Ib$  on the block matrix
$$\Ub(z):=\frac12\left(
                                                     \begin{array}{cc}
                                                       \Ab(z) & \Bb(z) \\
                                                       \Ab(-z) & \Bb(-z)
                                                     \end{array}
                              \right),\ \hbox{where} \ \
\Ub^\sharp(z):=\frac12\left(
                                                     \begin{array}{cc}
                                                       \Ab^\sharp(z) & \Ab^\sharp(-z) \\
                                                       \Bb^\sharp(z) & \Bb^\sharp(-z)
                                                     \end{array}
                                                   \right),$$
(which is the condition of being an unitary matrix for
$|z|=1$).
\par
Note that, given an orthogonal symbol $\Ab(z)$, the
alternating flip trick, used to construct the wavelet symbol in the scalar case, does not work in
this context. Indeed, the symbol $\Bb(z)=z\Ab^\sharp (-z)$ does
not verify the QMF equations unless $\Ab(z)$ and $\Ab(-z)$
commute.
\par
On the other hand, the orthogonality  conditions can be written in
terms of the subsymbols $\Ab_0(z),\ \Ab_1(z)$ and $\Bb_0(z),\
\Bb_1 (z)$. In fact, from (\ref{subsymbols}) we see that the
matrix $\Vb(z)$ 
$$\Vb(z):=\left(
                                                     \begin{array}{cc}
                                                       \Ab_0(z^2) & \Bb_0(z^2) \\
                                                       \Ab_1(z^2) & \Bb_1(z^2)
                                                     \end{array}
                              \right)=\left(
                                                     \begin{array}{cc}
                                                       1 & 1 \\
                                                       \frac1{z} & -\frac1{z}
                                                     \end{array}
                              \right)\Ub(z)$$
 is such that $\Vb^\sharp(z)\Vb(z)=2 \Ib$. Thus,
 the QMF equations take the equivalent form
\begin{equation}\label{QMFbis}
 \begin{array}{ccc}
&&\Ab_0^\sharp (z^2)\Ab_0(z^2)+\Ab_1^\sharp (z^2)\Ab_1(z^2)=2\Ib, \quad |z|=1, \\
&&\Ab_0^\sharp (z^2)\Bb_0(z^2)+\Ab_1^\sharp (z^2)\Bb_1(z^2)=\BZero, \quad \ |z|=1,\\
&&\Bb_0^\sharp (z^2)\Bb_0(z^2)+\Bb_1^\sharp (z^2)\Bb_1(z^2)=2\Ib, \quad
|z|=1.
 \end{array}
 \end{equation}
\par
Suppose now that we are given an orthogonal MCMRA  generated by some
orthonormal matrix  scaling function $\Fb$, with $\Gb$ as the
corresponding multichannel wavelet. Let us  consider any
vector-valued function $\hb\in L_2^r(\RR)$. From the nesting
properties of the spaces $\{V_j\}$ and $\{W_j\}$, it follows that
the approximation  $P_\ell\hb$ of $\hb$ in the space
$V_\ell$, $\ell\in \ZZ$, can be found in terms of the following
\emph{multichannel wavelet decomposition}:
$$P_\ell\hb=P_{\ell-L}\hb+Q_{\ell-1}\hb+Q_{\ell-2}\hb+\ldots+Q_{\ell-L}\hb,$$
where $L>0$ and $P_{\ell-L}\hb$, $ Q_{\ell-j}\hb$, $
j=1,\ldots,L,$ represent the orthogonal projections of $\hb$ to
the spaces $V_{\ell-L}, W_{\ell-j}, j=1,\ldots,L,$  respectively.
Analogously to the scalar case, we can derive a fast algorithm
which allows to compute all the projections by means of a
recursive scheme.
In fact we have that, if $P_j\hb\in V_j$ then $$P_j\hb=\sum_{k\in
\ZZ} \Fb(2^j\cdot -k)\cb_k^{(j)}$$ where
$$\cb_k^{(j)}=\langle\Fb(2^j\cdot -k),\hb\rangle=\int \Fb^H(2^jx -k)\hb(x)dx.$$
To compute the vector coefficient sequence $c^{(j-1)}\in \ell^r_2(\ZZ)$
connected to the representation of $\hb$ in the space $V_{j-1}$,
we make use of the refinement equation and get 
\begin{eqnarray*}
\cb_k^{(j-1)}&=& \langle\Fb(2^{j-1}\cdot -k),\hb\rangle=\langle\sum_{n\in\ZZ}\Fb(2(2^{j-1}\cdot -k)-n)\Ab_n,\hb\rangle\\
&=& \sum_{n\in\ZZ} \Ab_{n-2k}^T \langle\Fb(2^{j}\cdot -n),\hb\rangle=\sum_{n\in\ZZ}
\Ab_{n-2k}^T \cb_n^j,\quad k\in \ZZ.
\end{eqnarray*}
In the same way, using (\ref{eq:twoscalewavelet}), the wavelet vector coefficients
sequence $d^{(j-1)}\in \ell^r_2(\ZZ)$ connected to the
representation of $\hb$ in the space $W_{j-1}$ is obtained as
\begin{eqnarray*}
\db_k^{(j-1)}&=& \langle\Gb(2^{j-1}\cdot -k),\hb\rangle=\sum_{n\in\ZZ} \Bb_{n-2k}^T \cb_n^j,\quad k\in \ZZ.
\end{eqnarray*}
In summary, assuming $\ell=0$, that is $V_0$ as the initial space
of our representation, the \emph{vector decomposition formula} up
to the level $L>0$ reads as
\begin{eqnarray*}
\cb_k^{(j-1)}=\sum_{n\in\ZZ} \Ab_{n-2k}^T \cb_n^j, \quad \db_k^{(j-1)}=
\sum_{n\in\ZZ} \Bb_{n-2k}^T \cb_n^j,\quad k\in \ZZ,\, j=0,\ldots,L.
\end{eqnarray*}
Conversely, given the projections $P_{j}\hb$ and $Q_{j}\hb$, the
vector coefficient sequence $c^{(j+1)}$ connected to the
representation of $\hb$ in the space $V_{j+1}=V_{j}\oplus W_{j}$
is obtained by considering that
\begin{equation}\label{eq:rec1}
P_{j+1}\hb=\sum_{k\in\ZZ} \Fb(2^{j+1}\cdot -k)\cb_k^{(j+1)}\end{equation}
 and, on the other hand,
$$P_{j+1}\hb=P_{j}\hb+Q_{j}\hb=\sum_{n\in\ZZ} \Fb(2^{j}\cdot -n)\cb_n^{(j)}+\sum_{n\in\ZZ} \Gb(2^{j}\cdot -n)\db_n^{(j)}.$$
By invoking again the refinement equation on the right-hand side of the previous expression,
we have that
\begin{equation}\label{eq:rec2}
P_{j+1}\hb=\sum_{n\in\ZZ}\sum_{k\in\ZZ}\Fb(2^{j+1}-k)\Ab_{k-2n}\cb_n^{(j)}+\sum_{n\in\ZZ}\sum_{k\in\ZZ}\Gb(2^{j+1}-n)\Bb_{k-2n}\db_n^{(j)}.\end{equation}
A comparison between the two expressions (\ref{eq:rec1}) and (\ref{eq:rec2}) gives the
\emph{vector reconstruction formula}:
$$\cb_k^{(j+1)}=\sum_{n\in \ZZ}\left(\Ab_{k-2n}\cb_n^{(j)}+\Bb_{k-2n}\db_n^{(j)}\right),\quad k\in \ZZ, j=-L,\ldots,0.$$
For similar matrix wavelet decomposition and recontruction schemes see \cite{XiaSuter}.

\section{Multichannel wavelet construction}
The equation $\Ab_0^\sharp (z)\Ab_0(z)+\Ab_1^\sharp
(z)\Ab_1(z)=2\Ib$ is the starting point of the procedure that we
propose to construct the symbol $\Bb(z)$ of the multichannel
wavelet. More in detail, the procedure derives the matrix Laurent
polynomials $\Bb_0(z)$ and $\Bb_1(z)$, and thus
$\Bb(z)=\Bb_0(z^2)+z\Bb_1(z^2)$.
\par
 In the simple situation where $\Ab_0(z)$ and
$\Ab_1(z)$ are diagonal symbols, the multichannel wavelet
subsymbols $\Bb_0(z)$ and $\Bb_1(z)$  are constructed by the
repeated application of a scalar procedure: for each couple
$(a_0(z))_{ii},(a_1(z))_{ii}$, $i=1,\cdots, r,$ the two Laurent
polynomials $(b_0(z))_{ii},(b_1(z))_{ii}$ solution of the Bezout
identity are derived (see Lemma \ref{Lemma:nocommonzero}).
\par
In the general situation, the strategy is based on the existence
of two matrix symbols $\Db_0(z)$, $\Db_1(z)$ such that the positive definite Hermitian matrix
\begin{equation}\label{def:G}
\Gb(z)=:\left(%
\begin{array}{cc}
  \Ab_0(z) & \Db^\sharp_0(z) \\
  \Ab_1(z) & \Db^\sharp_1(z)
\end{array}%
\right)
\end{equation}
satisfies
\begin{eqnarray}\label{prop:GG}
\det \Gb(z)=1,\quad \quad \Gb^\sharp(z)\Gb(z)
=\left( \begin{array}{cc}
  2\Ib &  \BZero \\
  \BZero&  \Db_0(z)\Db_0^\sharp(z)+\Db_1(z)\Db_1^\sharp(z)
\end{array}%
\right),
\end{eqnarray}
as proved in the following proposition.
\begin{proposition} Let $\Ab_0(z)$ and $\Ab_1(z)$ be such that $\Ab_0^\sharp (z)\Ab_0(z)+\Ab_1^\sharp
(z)\Ab_1(z)=2\Ib$. Suppose that there exist two matrix symbols $\Db_0(z)$, $\Db_1(z)$  such that the positive definite Hermitian matrix
$\Gb(z)$ in (\ref{def:G}) has the properties (\ref{prop:GG}). 
Let $\Kb(z)$ be the spectral
factor of $\Db(z):=2\left(\Db_0(z)\Db_0^\sharp(z)+\Db_1(z)\Db_1^\sharp(z)\right)$.
Then  the  matrix symbols 
\begin{equation}\label{ris}
\Bb_0(z)=2\Db_0^\sharp(z)\left(\Kb^\sharp(z)\right)^{-1},\quad
\Bb_1(z)=2\Db_1^\sharp(z)\left(\Kb^\sharp(z)\right)^{-1},
\end{equation}
satisfy the two last
equations in (\ref{QMFbis}).
\end{proposition}
\begin{proof}
Since
$\Db(z)$ is
positive definite with determinant equal to $1$, its spectral
factor $\Kb(z)$ exists and satisfies
$\Db(z)=\Kb(z)\Kb^\sharp(z)$. 
Therefore, if we substitute (\ref{ris}) into
$$\Ab^\sharp_0(z)\Db^\sharp_0(z)+\Ab^\sharp_1(z)\Db^\sharp_1(z)=\BZero$$
we get 
$$
\Ab_0^\sharp (z^2)\Bb_0(z^2)+\Ab_1^\sharp (z^2)\Bb_1(z^2)=\BZero, \quad \ |z|=1,$$
which is the second equation in (\ref{QMFbis}). Finally, since 
$$\left(\Kb(z)\right)^{-1}\Db(z)\left(\Kb^\sharp(z)\right)^{-1}=\Ib,$$
by substituting (\ref{ris}) into it we end up with
$$
\Bb_0^\sharp (z^2)\Bb_0(z^2)+\Bb_1^\sharp (z^2)\Bb_1(z^2)=2\Ib, \quad
|z|=1$$
which is the  last equation in (\ref{QMFbis}). \end{proof}

For the  proof of the existence of the Hermitian matrix $\Gb(z)$ we
refer the reader to \cite{MicchelliSauer97a}. Its
actual
construction can be carried out as explained in the following subsections.

\subsection{Some matrix completion results}
We start by recalling a result
about Bezout identities

\begin{theorem}\label{Teo:Bezoutidentity}  \cite{LawMic00}
For any pair of Laurent polynomials $a_1(z)$ $a_2(z)$ the Bezout
identity
\begin{equation}\label{Bezoutidentity}
a_1(z)b_1(z)+a_2(z)b_2(z)=1
\end{equation}
is satisfied by a pair $b_1(z)$ $b_2(z)$ of Laurent polynomials if
and only if $a_1(z)$ and $a_2(z)$ have no common zeros. Moreover,
given one particular pair of solutions $b^*_1(z)$, $b^*_2(z)$ the
set of all solutions of (\ref{Bezoutidentity}) is of the form
$$
b_1(z)=b^*_1(z)+p(z)a_2(z),\quad b_2(z)=b^*_2(z)-p(z)a_1(z)
$$
where $p(z)$ is any Laurent polynomial.
\end{theorem}
\par
As to the common zeros of the subsymbols, the
following result holds true.
\begin{lemma}\label{Lemma:nocommonzero}
The subsymbols associated with the diagonal entries of an
orthogonal symbol $\Ab(z)$, that is the subsymbols
$(a_0(z))_{ii},\ (a_1(z))_{ii},\ i=1,\cdots,r,$ have no common
zeros.
\end{lemma}
\begin{proof} The result follows from the first equation in (\ref{QMFbis}). In
fact the existence of a common zero for $(a_0(z))_{ii},\
(a_1(z))_{ii},\ i=1,\cdots,r$ contradicts
$$
(a_0(z^{-1}))_{ii}(a_0(z))_{ii}+(a_1(z^{-1}))_{ii}(a_1(z))_{ii}=2,\quad
i=1,\dots,r,
$$
a relation  satisfied by the diagonal symbols.\qed \end{proof}
\par
We continue with a "completion" result.
\begin{theorem}\label{Teo:Completation(M=1N>=2)}
Let $a_i(z),\ i=1,\dots,n,$  Laurent polynomials with $n\ge 2$
with no common zeros. Then there exists a $n\times n$ matrix
Laurent polynomial $\Pb(z)$ whose first row is
$\left(a_1(z),\cdots,a_n(z)\right)$ such that
$$\det \Pb(z)=1,\quad z\in \CC\setminus\{0\}.$$
\end{theorem}
\begin{proof}
The proof is by induction. Let us start with $n=2$. For the
row vector $\left(a_1(z),a_2(z)\right)$ we construct the matrix
$\Pb(z)=\left(%
\begin{array}{cc}
  a_1(z) & a_2(z) \\
  -b_1(z) & b_2(z)
\end{array}%
\right)$ with $b_i(z),\ i=1,2$ being solutions of the Bezout
identity (\ref{Bezoutidentity}). Obviously, the matrix $\Pb(z)$ is
such that $\det \Pb(z)=1,\ z\in \CC\setminus\{0\}$. Next, assuming
that the theorem is true for $n-1$, we prove it for $n$. Thus,
given $\left(a_1(z),\cdots,a_n(z)\right)$ we first construct
${\bar \Pb}(z)$, $\det \bar \Pb(z)=1$, with the first row given by
$\left(a_1(z),\cdots,a_{n-1}(z)\right)$ then we construct the
block matrix
$$
\Pb(z):=\left(%
\begin{array}{ccccc}
   &  &  &  & a_n(z) \\
   &  &  &  & 0 \\
   & & {\bar P}(z) &  & \vdots \\
   &  &  &  & 0 \\
  0 & 0 & \cdots &0  & 1
\end{array}%
\right).
$$
This completes the proof. \qed \end{proof}

\par
With this result we are now able to describe a procedure for ''completing''
an $m\times n$ matrix Laurent polynomial with
$m\ge n$ and  $\mbox{rank}\, \Ab(z)=m$, by constructing an
$n\times n$ matrix Laurent
polynomial $\Pb(z)$, whose first $m$ rows agree with those of
$\Ab(z)$, such that
$$\det \Pb(z)=1,\quad z\in \CC\setminus\{0\}.$$

The existence of such polynomial has been formerly proved in \cite{MicchelliSauer97a}.
\par

\smallskip 

We start by discussing how to write a procedure providing, for a
given Laurent polynomial vector $\ab_{n}(z)$ of length $n\ge 2$, a
Laurent polynomial matrix $\Pb(z)$ having the first row given by
$\ab_n(z)=\left(a_1(z),\cdots,a_{n}(z)\right)$ and satisfying
 $\det \Pb(z)=1,\ z\in \CC\setminus \{0\}$.
 \par
 For $a_1(z),a_2(z)$, let
[$b_1(z),b_2(z)$]=\verb"Bezout"($a_1(z),a_2(z)$) be the procedure
in Matlab-like notation providing the two Laurent polynomials
$b_1(z),b_2(z)$ solution of the Bezout identity.
\par
Next, let [$
\Pb_{n\times n}(z)$]=\verb"Basic_completion"($\ab_{n}(z)$) be the following
(recursive) procedure that, taking the vector $\ab(z)$ as input,
produces the matrix $\Pb_{n\times n}(z)$ as output.

\medskip \noindent
$\bullet$ [$\Pb_{n\times
n}(z)$]=\verb"Basic_completion"($\ab_{n}(z)$)

\begin{enumerate}\item[]
If $n=2$ let
[$b_1(z),b_2(z)$]=\verb"Bezout"($a_1(z),a_2(z)$) two Laurent
polynomials solution of the Bezout identity. Set $\Pb_{2\times
2}(z):=\left(
                    \begin{array}{cc}
                      a_1(z) & a_2(z) \\
                      -b_1(z) & b_2(z)
                    \end{array}
                  \right)$\,; \\
else ($n> 2$) use [$\Pb_{(n-1)\times
 (n-1)}(z)$]=\verb"Basic_completion"($\ab_{n-1}(z)$).

  Then with $\eb_{n-1}:=(0,\cdots, 1)^T\in \RR^{n-1}$, define
$$\Pb_{n\times n}(z):=\left(
                    \begin{array}{cc}
                      \Pb_{(n-1)\times (n-1)}(z) & a_n(z) \\
                      \BZero_{1\times(n-1)} & \eb_{n-1}
                    \end{array}
                  \right).$$
\end{enumerate}

\bigskip \noindent We continue by describing the (recursive) procedure [$\Pb_{m\times
m}$]=\verb"Completion"($\Ab_{n\times m}$). This procedure, for a
given matrix Laurent polynomial $\Ab_{n\times m}(z)$ with $m>n$,
constructs the Laurent polynomial matrix $\Pb(z)$ whose $n$ first
rows
 agree with those of $\Ab_{n\times m}$ and
 such that $\det \Pb(z)=1,\ z\in \CC\setminus \{0\}$.

\medskip \noindent
$\bullet$ [$\Pb_{m\times m}(z)$]=\verb"Completion"($\Ab_{n\times
m}(z)$)

\begin{enumerate}

\item Take $\bar \Ab_{(n-1)\times m}(z)$ the sub-matrix of
$\Ab_{n\times m}(z)$ made of its first $(n-1)$ rows\,;

\item If $n=2$ construct  [$\bar \Pb_{m\times
 m}(z)$]=\verb"Basic_completion"($\bar \Ab_{1\times m}(z)$)

else ($n>2$) construct  [$\bar \Pb_{m\times
 m}(z)$]=\verb"Completion"($\bar \Ab_{(n-1)\times m}(z)$)\,;

\item Compute the matrix $\Cb_{n\times m}(z):=\Ab_{n\times
m}(z){\bar \Pb_{m\times
 m}}^{-1}(z)$ having block structure
 $$
\Cb_{n\times m}(z)=\left(%
\begin{array}{cc}
  \Ib_{(n-1)\times (n-1)} & \BZero_{(n-1)\times (r+1)}\\
  \cb_{1\times (n-1)}(z)& \db_{1\times (r+1)}(z)
\end{array}%
\right)\,$$ where $r:=m-n$;

\item Construct the $(r+1)\times (r+1)$ matrix \\
$[\Db_{(r+1)\times
(r+1)}(z)]$=\verb"Basic_completion"($\db_{1\times (r+1)}(z)$)\,;

\item Use $\Db_{(r+1)\times (r+1)}(z)$ to construct the $m\times
m$ block matrix
$$\Vb_{m\times m}(z):=\left(
                    \begin{array}{cc|cc}
                      \Ib_{(n-1)\times (n-1)}(z)& && \BZero_{(n-1)\times (r+1)}
                      \\\hline
                     \begin{array}{c}  \cb_{1\times (n-1)}(z) \\ \BZero_{r\times (n-1)} \end{array} &&&
\Db_{(r+1)\times(r+1)}(z)
                    \end{array}
                  \right)$$
\item Set $\Pb_{m\times m}(z):=\Vb_{m\times m}(z)\bar \Pb_{m\times
m}(z)$\,.
\end{enumerate}
\par
With the help of the two previous recursive procedures, we are able
to sketch the algorithm for constructing multichannel wavelets.

\subsection{The MCW construction algorithm}\label{sec:MCWconstr}

\begin{enumerate}
\item From $\Ab(z)\in \ell_0^{r\times r}(\ZZ)$ extract the
sub-symbols $\Ab_0(z)$ and $\Ab_1(z)$;

\item If $\Ab_0(z)$ and $\Ab_1(z)$ are diagonal symbols construct $\Bb_0(z)$ and
$\Bb_1(z)$ by the repeated application of the procedure

[$(b_0(z))_{ii},(b_1(z))_{ii}$]=\verb"Bezout"($(a_0(z))_{ii},(a_1(z))_{ii}$),
$i=1,\cdots, r$

Then go to step 10;

\item Otherwise construct the $r\times 2r$ block matrix ${\widetilde \Ab^\sharp}(z)=\left(%
\begin{array}{ccc}
  \Ab_0^\sharp(z)& |& \Ab_1^\sharp(z)
\end{array}%
\right)$;

\item Use [$\Lb(z)$]=\verb"Completion"($\widetilde
\Ab^\sharp(z)$) to construct the
matrix $$\Lb(z)=:\left(%
\begin{array}{cc}
  \Ab^\sharp_0(z) & \Ab^\sharp_1(z) \\
  \Cb_0(z) & \Cb_1(z)
\end{array}%
\right)$$ so that
$$\det \left(\begin{array}{cc}
  \Ab_0(z) & \Cb^\sharp_0(z) \\
  \Ab_1(z) & \Cb^\sharp_1(z)
\end{array}%
\right)=1;$$

\item Compute the $r\times r$ matrix $\Rb(z)=\frac 12
\left(\Ab_0^{\sharp}(z)\Cb^{\sharp}_0(z)+\Ab_1^{\sharp}(z)\Cb^{\sharp}_1(z)\right)$;

\item Compute the $r\times r$ matrices
$\Db_i(z)=-\Rb^{\sharp}(z)\Ab^\sharp_i(z)+\Cb_i(z),\quad i=0,1$;

\item Compute the $r\times r$ positive definite matrix
$$\Db(z)=2\left(\Db_0(z)\Db_0^{\sharp}(z)+\Db_1(z)\Db_1^{\sharp}(z)\right);$$

\item Compute  the spectral factorization of
$\Db(z)$ such that $$\Db(z)=\Kb(z)\Kb^{\sharp}(z),$$ for example using the algorithm described in \cite{JezekKucera85};

\item Set $\Eb(z)=\Kb^{-1}(z)$ and construct the $r\times r$
matrices $$\Bb_i(z)=2\Db_i^\sharp(z)\Eb^\sharp(z),\quad i=0,1;$$

 \item Construct the wavelet symbol
 $$\Bb(z)=\Bb_0(z^2)+z\Bb_1(z^2).$$

\end{enumerate}
\par
Two remarks are worth to be made:
\begin{itemize}
\item For a proof of the correctness of all steps we refer to
\cite[Section 6]{MicchelliSauer97a}. Nevertheless, since \cite{MicchelliSauer97a} deals with
rank-1 matrix functions, the above construction has to be considered as an extention 
to the full rank case;

\item The algorithm produces correct results
even in case of diagonal subsymbols $\Ab_0(z)$ and $\Ab_1(z)$.
Though, to end up with exactly diagonal subsymbols $\Bb_0(z)$ and
$\Bb_1(z)$ we may need to multiply the result of the MCW algorithm
by a suitable permutation matrix $\Jb$.
\end{itemize}

\section{Numerical Examples}\label{sec:num}
The aim of this section is to consider two orthonormal full rank
refinable matrix functions with $r=2$ and $r=3$, respectively and
to use the MCW algorithm described in the subsection \ref{sec:MCWconstr} to construct the
corresponding multichannel wavelets. In particular, while for the
case $r=2$ the positive definite full rank symbol that we start with is
the one considered in \cite{ContiCotroneiSauer08}, the case $r=3$
is here derived from scratch.

\subsection{A two-channel example}\label{sec:numr2}
Let $\Cb(z)$ be the symbol of the interpolatory full rank vector subdivision scheme
first given in \cite{ContiCotroneiSauer08}
\begin{eqnarray*}
\Cb(z) = \left(
  \begin{array}{cc}
    c_{11}(z) & c_{12}(z) \\ c_{12}(z^{-1}) & c_{22}(z)
  \end{array} \right)
\end{eqnarray*}
with
\begin{eqnarray*}
c_{11}(z)&=&\frac 12\,{\frac { \left( z+1 \right) ^{2}}{z}}\qquad  \mbox{(linear B-spline symbol)},\\
c_{22}(z)&=&-\frac1{16}\frac{(z^2-4z+1)(z+1)^4}{z^3}\qquad   \mbox{(4-point scheme symbol)},\\
c_{12}(z)&=&\lambda\,z\,(z^2-1)^3,
\end{eqnarray*}
satisfying $\Cb(1)=2\Ib,\ \Cb(z)+\Cb(-z)=2\Ib$.
 In order to assure positive definiteness, the parameter $\lambda$ is taken in the interval
$(-\frac 1{32}\sqrt{3},\frac 1{32}\sqrt{3})$.
Fixing, for example, $\lambda=\frac 1{20}$, the symbol $\Ab(z)$ of the associated  full rank orthonormal
refinable function constructed with the algorithm given in \cite{ContiCotroneiSauer08}
has elements
{\small
\begin{eqnarray*}
  a_{11}(z)&=&  1.081604742 + 0.7776021479\,z + {\displaystyle
    0.2165957837z^{-1}}- {\displaystyle 0.08257171989z^{-2}}\\&&
  + {\displaystyle 0.004835070286z^{-3}}
  +
  {\displaystyle 0.0009670018466z^{-4}} + {\displaystyle
    0.0009670219776z^{-5}}, \\
  a_{12}(z)&=& - 0.001208777472 + 0.001208777472\,z^{2} + 0.001208777472\,z -
  {\displaystyle 0.001208777472z^{-1}}, \\
  a_{21}(z)&=&   - 0.02117594497 - 0.005136018632\,z - {\displaystyle
    0.07331687185z^{-1}}  + {\displaystyle 0.2747671662z^{-2}
  }\\&& - {\displaystyle 0.0679559177z^{-3}}  -
  {\displaystyle 0.2535912212z^{-4}} + {\displaystyle
    0.1464088082z^{-5}},\\
  a_{22}(z)&=& 0.3169890016 + 0.6830110222\,z^{2} + 1.183011034\,z
  -  {\displaystyle 0.1830110103z^{-1}}.
\end{eqnarray*}
}
The  plot of the corresponding matrix scaling function is given in Fig. \ref{fig:G2x2}

\begin{figure}[!t]
\begin{center}
\includegraphics[width=10cm]{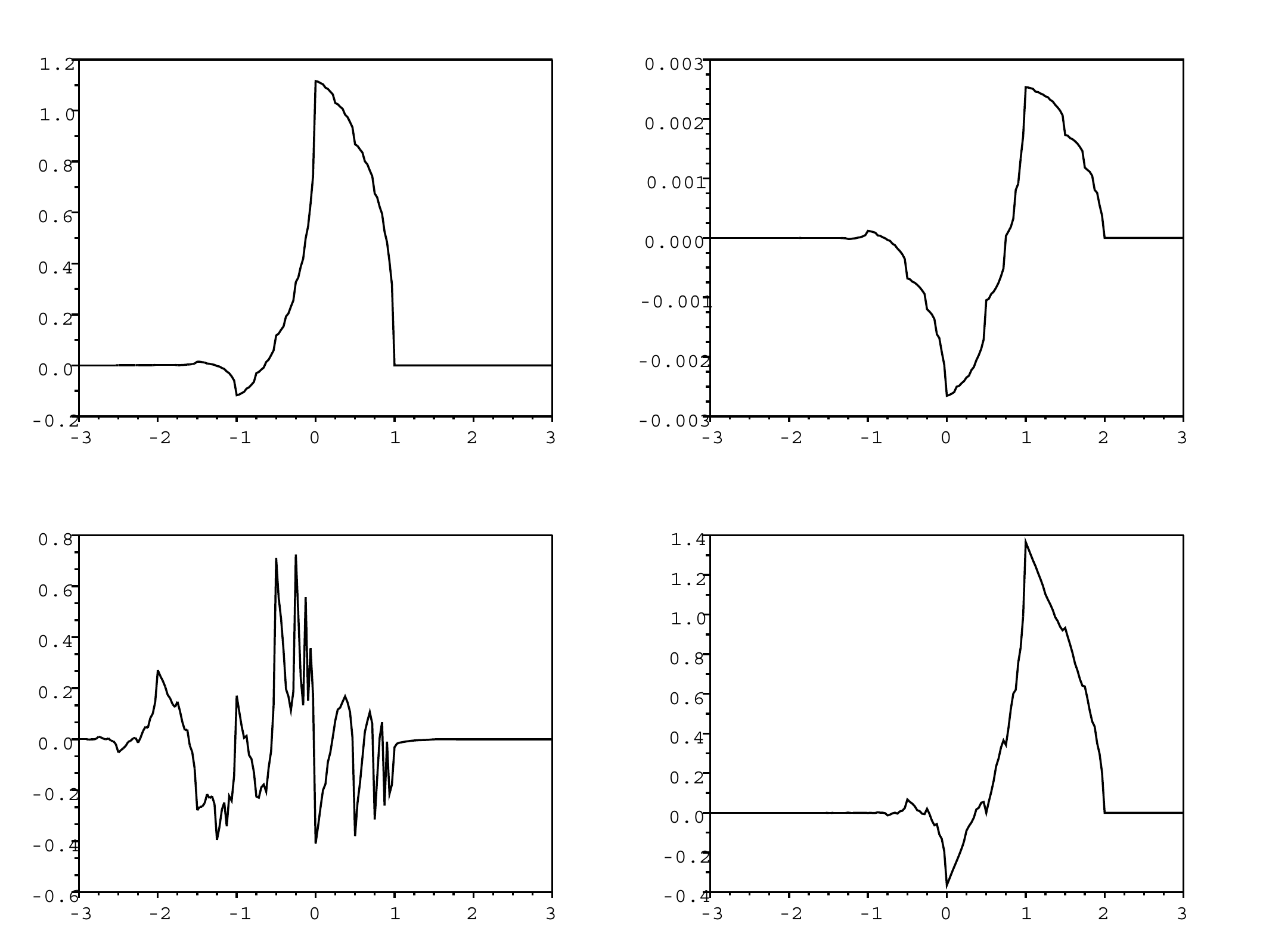}
\caption{Orthonormal matrix scaling function with $r=2$}
\label{fig:G2x2}
\end{center}
\end{figure}

\begin{figure}[!t]
\begin{center}
\includegraphics[width=10cm]{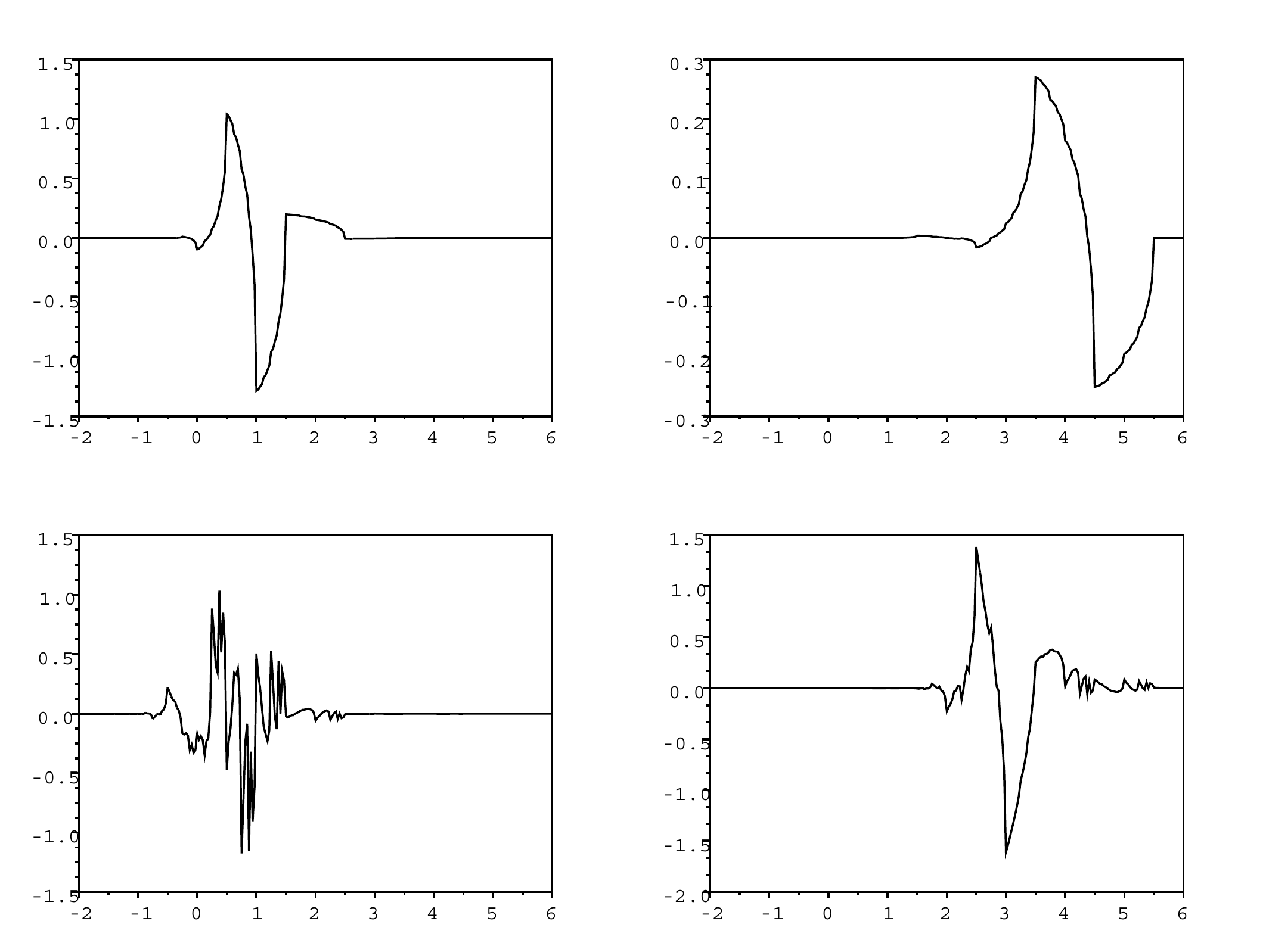}
\caption{Orthornormal 2-channel wavelet}
\label{fig:WG2x2}

\end{center}
\end{figure}

Using the MCW Algorithm  we find that the  symbol $\Bb(z)$ of the
corresponding orthonormal two-channel wavelet has the coefficients
$\Bb_k$, $k=-2,\ldots,7,$ given in Table~\ref{tab:WG2x2}. The plot
of the wavelet is represented in Fig.~\ref{fig:WG2x2}.

\begin{table}[!t]
\renewcommand{\arraystretch}{1.3}
 \caption{Coefficients of the 2-channel orthonormal wavelet}   \label{tab:WG2x2}
    \centering
        \begin{tabular}{cc||cc}\hline
      $k$ & $\Bb_k$ &$k$ & $\Bb_k$ \\\hline\hline
      $-2$&        $ \left(\begin{array}{cc}
0.8140199&0\\ -0.0014406&0  \end{array}\right)$&$3$ &
 $ \left(  \begin{array}{cc}  -0.0044401&0.0519219\\ -0.0000359&0.331059 \end{array}\right)$ \\[0.3cm]
$-1$ & $ \left(\begin{array}{cc} -1.1322992&0\\ 0&0    \end{array}\right)$ & $4$ &
 $ \left(\begin{array}{cc}     0.0009821&0.2546764\\ -0.0000192&0.1253084 \end{array}\right)$\\[0.3cm]
$0$ & $ \left(\begin{array}{cc}    0.1915220&0.0032939\\ 0.0037742&-0.0000058   \end{array}\right)$ & $5$ &
$ \left(\begin{array}{cc}   0.0007061&0.1209272\\ -0.0000047&0.0153323 \end{array}\right)$\\[0.3cm]
$1$ & $ \left(\begin{array}{cc}     0.1360351&-0.0002724\\ -0.0018743&0.6524543 \end{array}\right)$&$6$ &
$ \left(\begin{array}{cc}       0&-0.2427089\\ 0&0.0047518 \end{array}\right)$\\[0.3cm]
$2$ & $ \left(\begin{array}{cc}  -0.0065259&-0.0133465\\ -0.0003995&-1.1300526  \end{array}\right)$&$7$ &
 $ \left(\begin{array}{cc}     0&-0.1744916\\ 0&0.0011525 \end{array}\right)$\\
\hline
        \end{tabular}
\end{table}

\subsection{A three-channel example}\label{sec:numr3}
\begin{table}[!t]
\renewcommand{\arraystretch}{1.3}
 \caption{Coefficients of the 3-channel orthonormal matrix scaling function}   \label{tab:G3x3}
    \centering
        \begin{tabular}{cc||cc}\hline
      $k$ & $\Ab_k$ &$k$ & $\Ab_k$ \\ \hline\hline
      $-8$&        $ \left(\begin{array}{ccc} 0&0&0\\ 0&0&0\\ 0.04&0.03125&0    \end{array}\right)$&      $-2$ & $ \left(\begin{array}{ccc} -0.0715410&0.0216501&0\\
 0.2697914&-0.0206268&0\\ -0.04&-0.03125&0  \end{array}\right)$ \\  [0.8cm]
$-7$ & $ \left(\begin{array}{ccc}0&0&0\\ 0&0&0\\ -0.04&-0.03125&0\end{array}\right)$ & $-1$ & $ \left(\begin{array}{ccc}  0.2874359&0.0253824&0\\ -0.0756559&-0.1655344&0\\ 0.04&0.03125&0       \end{array}\right)$\\ [0.8cm]
$-6$ & $ \left(\begin{array}{ccc}  -0.0022120&-0.0017281&0\\ 0.0000731&0.0000571&0\\ -0.12&-0.09375&0      \end{array}\right)$ & $0$ & $ \left(\begin{array}{ccc}  1.0740925&-0.0422890&0\\ -0.0129914&0.3501563&0\\ 0&0&1      \end{array}\right)$\\[0.8cm]
$-5$ & $ \left(\begin{array}{ccc}    0.0087982&0.0060117&0\\ 0.1481828&-0.0017879&0\\ 0.12&0.09375&0   \end{array}\right)$&$1$ & $ \left(\begin{array}{ccc}0.7231598&-0.0073621&0\\ -0.0053020&1.1683135&0\\ 0&0&1\end{array}          \right)$\\[0.8cm]
$-4$ & $ \left(\begin{array}{ccc}  -0.0003396&0.0004183&0\\ -0.2568731&0.0059962&0\\ 0.12&0.09375&0       \end{array}\right)$&$2$ & $ \left(\begin{array}{ccc}        0&0.0219488&0\\ 0&0.6644173&0\\ 0&0&0\end{array}\right)$\\[0.8cm]
$-3$ & $ \left(\begin{array}{ccc} -0.0193938&-0.0240319&0\\ -0.0672249&-0.0009911&0\\ -0.12&-0.09375&0         \end{array}\right)$& & \\
\hline
        \end{tabular}
\end{table}

We now construct a $3\times 3$ positive definite parahermitian
interpolatory symbol
$$\Cb(z)=\left(
  \begin{array}{ccc}
    c_{11}(z) & c_{12}(z) & c_{13}(z) \\ c_{12}(z^{-1}) & c_{22}(z) &
    c_{23}(z) \\ c_{13}(z^{-1}) & c_{23}(z^{-1}) & c_{33}(z)
  \end{array}\right)$$
satisfying $\Cb(1)=2\Ib,\ \Cb(z)+\Cb(-z)=2\Ib$. It means that, as in the previous example, the
diagonal elements should be the symbols of interpolatory scalar
schemes, with factor $(z+1)$ of order $m_1,m_2,m_3$, respectively. The off-diagonals should certainly contain a
$(z^2-1)$ factor and must satisfy $c_{ij}(z)=-c_{ij}(-z)$ (see \cite{ContiCotroneiSauer07}).  We
take on the diagonal the Deslaurier-Dubuc filters with
$m_1=2,m_2=4,m_3=2$, that is:
$$c_{11}(z)=\frac 12 \frac {(z+1)^2}{z},\ c_{22}(z)=-\frac 1{16}\frac{(z^2 - 4z + 1)(z + 1)^4}{z^3},\ c_{33}(z)=c_{11}(z).$$

Since $(z+1)^2$ is a common factor, repeating the consideration
done in \cite{ContiCotroneiSauer08}, we require the following
symbols on the off-diagonal
$$
c_{12}(z)=\lambda_1 \,z \, (z^2 - 1)^3,\,
c_{13}(z)=\lambda_2 \,z \,(z^2 - 1)^4,\,
c_{23}(z)=\lambda_3 \,z \,(z^2 - 1)^4.
$$
The following parameter choice
$$\lambda_1=1/20,\, \lambda_2=1/50,\,\lambda_3=1/64$$
gives positive definiteness. In this case, the canonical spectral
factor $\Ab(z)$, symbol of the orthonormal refinable function
$\Fb$, has the coefficients given in Table \ref{tab:G3x3}. The
plot of the corresponding scaling function is represented in
Fig.~\ref{fig:G3x3}.

\begin{table}[!t]
\renewcommand{\arraystretch}{1.3}
 \caption{Coefficients of the 3-channel orthonormal matrix wavelet}   \label{tab:WG3x3}
    \centering
        \begin{tabular}{cc||cc}\hline
      $k$ & $\Bb_k$ &$k$ & $\Bb_k$ \\ \hline\hline
      $-2$&        $ \left(\begin{array}{ccc}  -0.0031659&0&0\cr 0.0001046&0&0\cr 0&0&0      \end{array}\right)$&      $5$ & $ \left(\begin{array}{ccc} 0.1684248&0.0212803&0.0595696\cr -0.0017438&0.0229928&0.1406681\cr -0.0001092&-0.0004847&-0.0160288        \end{array}\right)$ \\  [0.8cm]
$-1$ & $ \left(\begin{array}{ccc}   0&0&0\cr -0.6416072&0&0\cr 0&0&0   \end{array}\right)$ & $6$ & $ \left(\begin{array}{ccc} -0.0017816&-0.0078392&-0.1305600\cr -0.0009269&0.0105339&0.0451634\cr 0&0&-0.0026508            \end{array}\right)$\\ [0.8cm]
$0$ & $ \left(\begin{array}{ccc}0.0297628&-0.7705220&-0.0001305\cr 1.1269644&0.0254539&0\cr 0.0367180&0.0169580&-0.9720037            \end{array}\right)$ & $7$ & $ \left(\begin{array}{ccc}    -0.0012167&-0.0056664&-0.0773112\cr -0.0015390&-0.0062807&-0.1222042\cr 0&0&0.0026508        \end{array}\right)$\\[0.8cm]
$1$ & $ \left(\begin{array}{ccc} -0.0940031&1.1447719&0\cr -0.3500526&-0.0169113&-0.0264419\cr -0.0367180&-0.0169580&0.9720037            \end{array}\right)$&$8$ & $ \left(\begin{array}{ccc}   -0.0000291&-0.0001188&0.0406917\cr -0.0008803&-0.0035956&-0.0543552\cr 0&0&0   \end{array}          \right)$\\[0.8cm]
$2$ & $ \left(\begin{array}{ccc} -0.2583832&-0.2560664&-0.0751903\cr -0.1287475&-0.0115762&0.0489688\cr -0.0096495&-0.0033517&-0.0409056            \end{array}\right)$&$9$ & $ \left(\begin{array}{ccc}   0&0&0.0295229\cr 0&0&0.0373432\cr 0&0&0      \end{array}\right)$\\[0.8cm]
$3$ & $ \left(\begin{array}{ccc}   -0.0908757&-0.1606346&0.0018325\cr -0.0045318&-0.0170932&-0.0567853\cr 0.0096495&0.0033517&0.0409056       \end{array}\right)$&   $10$  & $\left(\begin{array}{ccc}    0&0&0.0007057\cr 0&0&0.0213611\cr 0&0&0         \end{array}\right)$\\[0.8cm]
$4$ & $\left(\begin{array}{ccc} 0.2512676&0.0347952&0.1508696\cr 0.0029602&-0.0035235&-0.0337223\cr 0.0001092&0.0004847&0.0160288         \end{array}\right)$& & \\
\hline
        \end{tabular}
\end{table}

\begin{figure}[!t]
\begin{center}
\includegraphics[width=12cm]{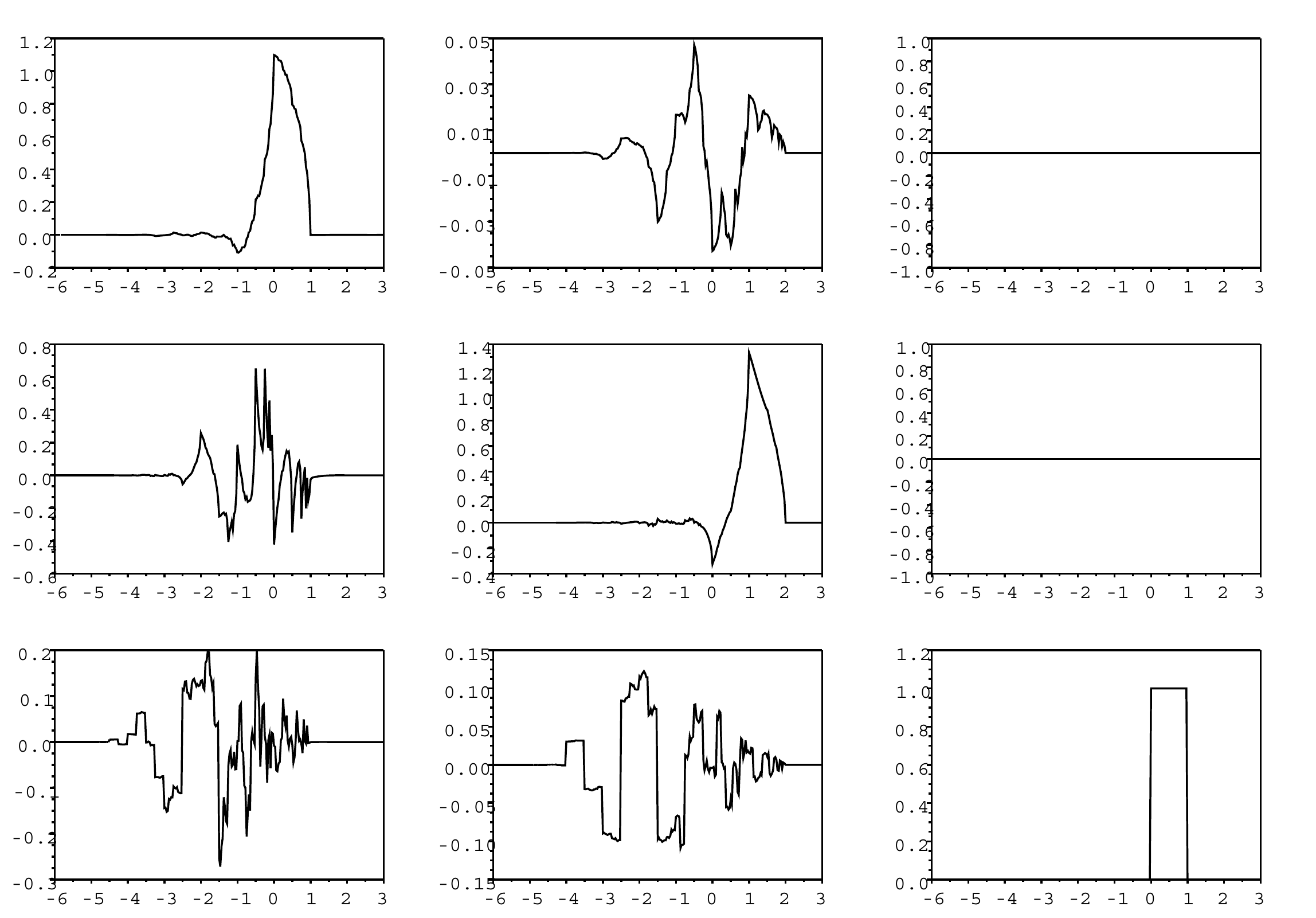}
\caption{Orthonormal matrix scaling function with $r=3$}
\label{fig:G3x3}
\end{center}
\end{figure}

\begin{figure}[!t]
\begin{center}
\includegraphics[width=12cm]{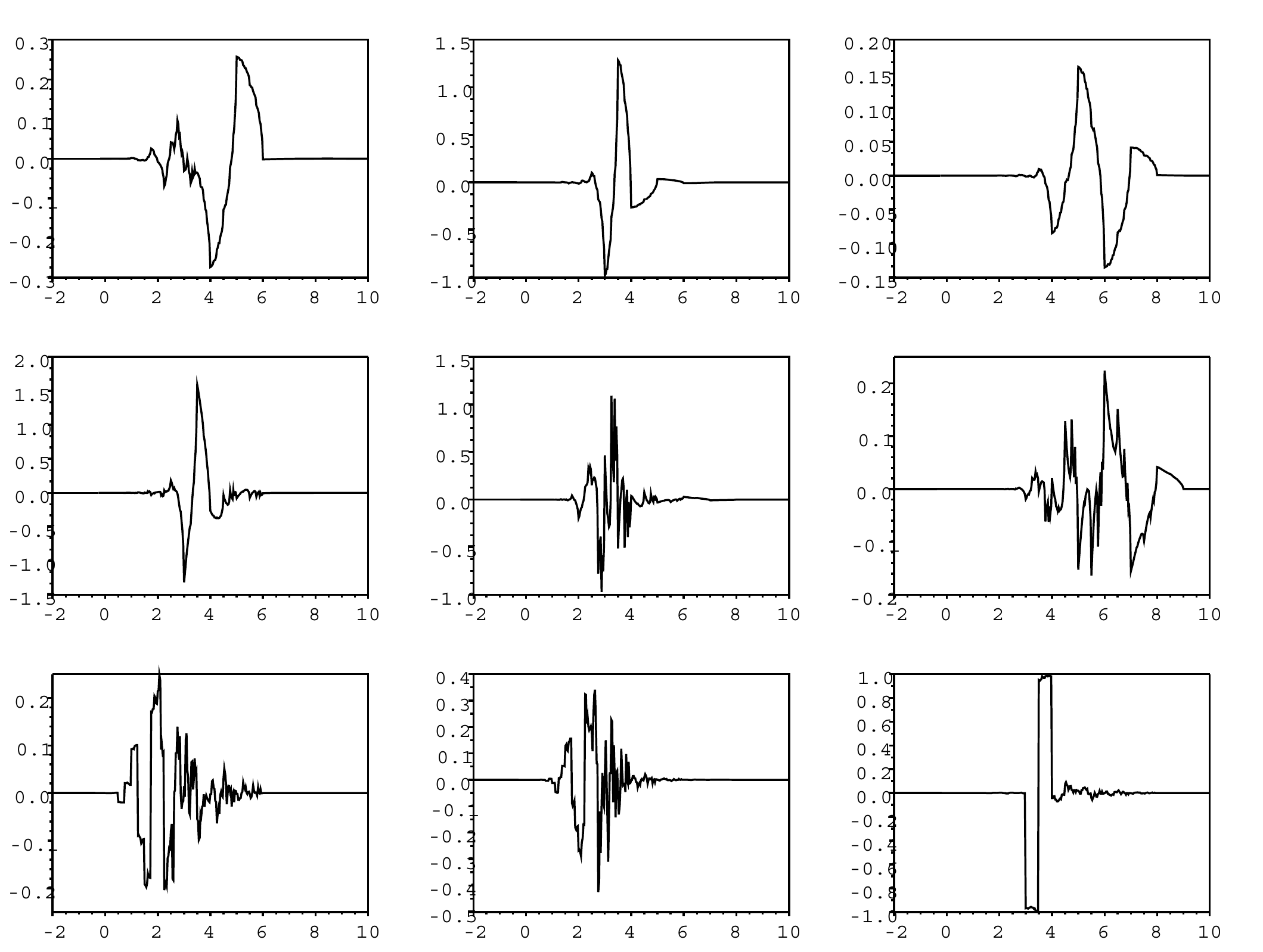}
\caption{Orthonormal 3-channel wavelet}
\label{fig:WG3x3}

\end{center}
\end{figure}

The MCW algorithm applied to this symbol, gives the orthonormal wavelet whose symbol coefficients are
given in Table \ref{tab:WG3x3} and whose plot is represented in Fig. \ref{fig:WG3x3}.

\section{Conclusions}
This paper discusses a way to construct orthonormal multichannel
wavelets from convergent full rank orthogonal symbols. As to our
knowledge, the explicit algorithm for multichannel wavelet
construction (the MCW construction algorithm) here derived, based on the strategy for rank-1
filters given in \cite{MicchelliSauer97a}, is the first algorithm proposed so
far in the literature.\par
As already pointed out, multichannel wavelets provide an effective
tool for the analysis of multichannel signals, that is
vector-valued signals whose components come from different sources
with possible intrinsic correlations, for example seismic waves,
brain activity (EEG/MEG) data, financial time series, color
images. Many multichannel signals exhibit a high correlation which can be revealed and exploited by a  MCW analysis,
with filters suitably tailored to the specific data and application.
Future investigations include the construction of ''data-adapted'' MCW bases and  their application in many problems where vector-valued signals have to be processed for compression, denoising, etc.



\begin{thebibliography}{99.}%

\bibitem{BacchelliCotroneiSauer02a}
S.~Bacchelli, M.~Cotronei, T.~Sauer, Wavelets for multichannel signals,
  Adv. Appl. Math., \textbf{29},  581--598 (2002)

\bibitem{ContiCotroneiSauer08}
C.~Conti, M.~Cotronei, T.~Sauer, Full rank interpolatory subdivision schemes:
Kronecker, filters and multiresolution, J. Comput. Appl. Math., \textbf{233~(7)}, 1649--1659 (2010)


\bibitem{ContiCotroneiSauer07S}
C.~Conti, M.~Cotronei, T.~Sauer, Full rank positive matrix symbols:
  interpolation and orthogonality, BIT,   \textbf{48}, 5--27 (2008)

\bibitem{ContiCotroneiSauer07}
C.~Conti, M.~Cotronei, T.~Sauer, Interpolatory vector subdivision schemes, in:
  A.~Cohen, J.~L. Merrien, L.~L. Schumaker (eds.), Curves and Surfaces, Avignon
  2006, Nashboro Press, (2007)

\bibitem{CotroneiSauer07}
M.~Cotronei, T.~Sauer, Full rank filters and polynomial reproduction, Comm.
  Pure Appl. Anal., \textbf{6}, 667--687 (2007)

\bibitem{Fowler} J.E. Fowler, L. Hua, Wavelet Transforms for Vector Fields Using Omnidirectionally Balanced Multiwavelets, IEEE Transactions on Signal Processing,  \textbf{50}, 3018--3027 (2002)

\bibitem{GHM} J. S. Geronimo, D. P. Hardin,  P. R. Massopust, Fractal functions
and wavelet expansions based on several scaling functions, J. Approx. Theory, \textbf{78~(3)},  373--401 (1994)

\bibitem{Keinert}
F.Keinert, Wavelets and Multiwavelets, Chapman \& Hall/CRC, (2004)

\bibitem{JezekKucera85}
 J.~Je{\v z}ek and V.~Ku{\v c}era, Efficient algorithm for matrix spectral
   factorization, Automatica, \textbf{21},  663--669 (1985)

\bibitem{LawMic00} Lawton, W. M.; Micchelli, C. A., Bezout identities
with inequality constraints.  Vietnam J. Math.  \textbf{28~(2)},
97--126 (2000)


\bibitem{MicchelliSauer97a}
C.~A. Micchelli, T.~Sauer, Regularity of multiwavelets, Adv. Comput. Math.,
  \textbf{7~(4)},  455--545 (1997)

\bibitem{Strela} G. Strang, V. Strela, Short wavelets and matrix dilation equations,
IEEE Trans. Signal Process.,  \textbf{43~(1)}, 108--115 (1995)

\bibitem{Xia} X.G. Xia, Orthonormal matrix valued wavelets and matrix Karhunen--Lo\`eve expansion. In: A. Aldroubi, E. B. Lin (eds.) Wavelets, multiwavelets, and their applications, Contemporary
Mathematics, \textbf{216},  159--175. Providence, RI: American Mathematical Society (1998)


\bibitem{XiaSuter} X.G. Xia, B. Suter, Vector-valued wavelets and vector filter banks, IEEE Trans.
Signal Process., \textbf{44},  508--518 (1996)

\bibitem{Walden} A. T. Walden, A. Serroukh, Wavelet analysis of matrix-valued time-series,
Proc. R. Soc. Lond. A, \textbf{458},  157--179 (2002)


\end{thebibliography}
\end{document}